\documentclass[preprint,authoryear,11pt]{elsarticle}

\usepackage{latexsym}
\usepackage{amsmath}
\usepackage{dsfont}
\usepackage{amsthm}
\usepackage{amsfonts}
\usepackage{amssymb}
\usepackage{psfrag}
\usepackage{graphicx}
\usepackage{verbatim} 
\usepackage{color}
\usepackage{booktabs}

\usepackage{tikz}
\usepackage{pgfplots}
\pgfplotsset{compat = newest}

\usepackage{subcaption}

\usepackage[margin=1in]{geometry}
\allowdisplaybreaks

\usepackage{enumerate}

\usepackage[margin=1in]{geometry}
\allowdisplaybreaks

\renewenvironment{description}[1][0pt]
  {\list{}{\labelwidth=0pt \leftmargin=#1
   }}
  {\endlist}

\allowdisplaybreaks

\newtheorem{theorem}{Theorem}

\newtheorem{lemma}[theorem]{Lemma}

\numberwithin{equation}{section} 
\numberwithin{theorem}{section}

\begin{document}
	\begin{frontmatter}
	
	\title{	An Application of Fractional Calculus to Column Theory}
		
		\author[jv]{Jos\'{e} Villa-Morales\corref{co1}} 
		\ead{jvilla@correo.uaa.mx}
		\address[jv]{Departamento de Matem\'{a}ticas y F\'{\i}sica,
			Universidad Aut\'onoma de Aguascalientes,
			Av. Universidad 940, C.P. 20100 Aguascalientes, Ags., M\'exico}
			
			\author[mr]{Manuel Ram\'irez-Aranda} 
		\ead{manuel.ramirez@edu.uaa.mx}
		\address[mr]{Departamento de Matem\'{a}ticas y F\'{\i}sica,
			Universidad Aut\'onoma de Aguascalientes,
			Av. Universidad 940, C.P. 20100 Aguascalientes, Ags., M\'exico}

		\cortext[co1]{Corresponding author}
		
\begin{abstract}

In this article, we employ a fractional version of the radius of curvature in Euler's equation for column buckling, enabling us to derive a fractional differential equation in the Caputo sense. We solve this equation and demonstrate that for certain values of the fractional parameter, there exists a critical buckling force. Additionally, we provide a numerical scheme for accurately approximating this critical force.
\end{abstract}

\end{frontmatter}
	
\section{Introduction}

The analysis of buckling in materials, particularly in columns, is of crucial importance in applied mathematics due to its numerous practical applications. This study is especially relevant in the construction industry, where the buckling of columns made of steel, aluminum, or plastic materials plays a key role. Similarly, this area of research is also significant in other sectors, such as aviation and the automotive industry, among others. For further exploration of this topic, it is recommended to consult sources \cite{domokos1997constrained, wang2004exact,luo2004equilibrium,jena2019novel} and the references included therein. 

In the study of column buckling, a pivotal aspect is the determination of the maximum force a column can endure prior to buckling, see \cite{cocskun2009determination}. This article delves into the investigation of this critical force, employing an analysis of a fractional differential equation in the Caputo framework, which we describe below.

Suppose we have a beam of length $l$ with one flexible end resting on the ground and a force $P$ is applied to the other flexible end. Under certain assumptions on the beam, the well known Euler equation (see \cite{Spi}) says that at a point $x$ the external moment $M(x)$ of the deflected beam is inversely proportional to the ratio of curvature $\rho (x)$. In symbols
\begin{equation}
M(x)=\frac{EI}{\rho(x)}, \label{EEq}
\end{equation}
where $E$ is the Young’s modulus of the material and $I$ denotes the second moment of inertia of the cross-section area. 
The displacement of the beam at point $x \in [0,l]$ is represented by $y(x)$, which corresponds to the elastic curve of the beam (refer to Figure \ref{FigB}).
\begin{figure}[ht]
\centering

\begin{tikzpicture}[scale=1]

\draw[-,line width = 0pt] (0,0) -- (2,0.5);
\draw[-] (0,5) -- (2,5.5);
\draw[dashed,black,-] (1,0.35) -- (1,5.2);
\draw[<-,very thick] (1,5.25) -- (1,6);
\draw[black] (1,6.25) node {\scriptsize $P$};
\draw[dashed,->] (-0.5,3.2) -- (1.46,3.4);
\draw[dashed,black,->] (0.8,2.5) -- (0.8,3.4);
\draw[dashed,black,-] (1,3.4) -- (1.45,3.4);
\draw[teal] (0.8,2.2) node {\scriptsize $x$};
\draw[dashed,black,-] (0.8,0.2) -- (0.8,2);
\draw[dashed,black, xshift=0cm] plot [smooth, tension=1] coordinates { (1,0.25) (1.5,2.75) (1,5.25)};
\draw[black, xshift=0cm, line width = 3pt] plot [smooth, tension=1] coordinates {(2,0.5) (2.5,3) (2,5.5)};
\draw[black, xshift=0cm, line width = 3pt] plot [smooth, tension=1] coordinates {(0,0) (0.5,2.5) (0,5)};
\draw[<-] (1.65,3.4) arc (90:-50:0.34cm);
\draw[teal] (2,2.5) node {\scriptsize $M(x)$};
\draw[teal] (1.85,1.5) node {\scriptsize $y(x)$};
\draw[teal] (-0.3,3.6) node {\scriptsize $\rho(x)$};
\draw[teal] (-0.6,1.6) node {\scriptsize $C(x)$};
\draw[dashed,thick,-] (1.45,3.3) arc (360: -40:2cm);
\draw[->] (1,0.25) -- (3,0.25);
\draw[black] (3.3,0.2) node {\scriptsize $y$};
\draw[black] (2.9,3) node {\scriptsize $l$};
\end{tikzpicture}
\caption{The elastic curve $y(\cdot)$ and the osculating circle $C(\cdot)$.} \label{FigB}
\end{figure}
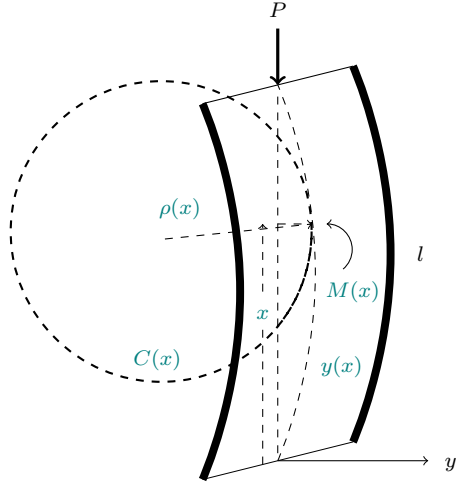 

Assuming a scheme like the one presented in Figure \ref{FigB}, we know (see \cite{TG}) that the external moment is given by, $M(x)=-P \, y(x)$. To introduce the fractional component, $\alpha\in (0,1]$, in the study of beam deflection, we work with the radius of curvature, $\rho (x)$. Indeed, through a thorough analysis utilizing the Caputo fractional derivative on the radius of curvature, as discussed in Section \ref{FRC}, we derive the fractional differential equation
\begin{equation}
EI\left(\frac{\Gamma(2-\alpha)}{x^{1-\alpha}}\cdot D_{0+}^{1+\alpha}y(x) \right)=-P \, y(x), \ \ x\in (0,l). \label{Cde}
\end{equation}

Currently, there is a wide variety of fractional derivatives available. Among the classical ones are the Riemann-Liouville, Grunwald-Letnikov, Riesz and Marchaud derivatives, to name a few, see \cite{diethelm2010analysis}, \citep{KST} or \cite{SKM}. In our study, we have opted for the Caputo derivative, primarily because it yields zero when applied to constant functions, and also because the order of the derivative is integer in the initial conditions of the differential equations we will be considering (see (\ref{fDEq2})).

When modeling physical phenomena using fractional differential equations, it is common to replace the classical derivative with a fractional one, as explained in \cite{shahroudi2023dynamic,lu2019analysis,cao2021numerical,liaskos2020implicit}. The main reason for this lies in the absence, within the fractional context, of an increment concept that would allow the deduction of the differential equation in fractional terms, see \cite{KALAK,fractalfract6110626}. In our study, as indicated in (\ref{apxlarc}), we use a linear approximation of the arc $S(x)$,  corresponding to the curvature radius $\rho (x)$, to obtain a fractional model of the buckling of a column.

The critical buckling force is given by
\begin{equation}
P=EI\,\Gamma(2-\alpha) \cdot \frac{s_{\bf 0}}{l^{2}}, \label{fcpandeo}
\end{equation}
where $s_{\bf 0}$ is the first positive zero of $y(x)$, the solution of (\ref{Cde}). In particular, if $\alpha=1$ then $s_{\bf 0}= \pi^{2}$, which coincides with the classic case, see \cite{TG}. Clearly, the main complexity lies in the calculation of $s_{\mathbf{0}}$. According to our literature review, there are no precedents of studies that have identified positive roots for solutions of fractional differential equations in the Caputo context. Beyond its varied practical applications, we believe our work introduces new techniques that could be valuable for advancing other disciplines of knowledge. In the study of solutions to ordinary differential equations solved using power series, for example

To determine $s_{\bf 0}$, or more precisely, to find an approximation of $s_{\bf 0}$, we note that the function $y(x)$ is the uniform limit, over compact intervals, of certain polynomials. Assuming there is a polynomial $p_{n_{0}}$ that meets specific conditions (see Assumption A), we demonstrate that $y(x)$ has a positive root, as detailed in Theorem \ref{Th1}. Furthermore, if we impose additional regularity conditions on $p_{n_{0}}$ (see Assumption B), then for a given error $\varepsilon$, we can obtain an approximation of $s_{\bf 0}$ that differs by less than $\varepsilon$ units, as outlined in Theorem \ref{Th2}. Thus, we manage to obtain an approximation of the critical force $P$, as seen in equation (\ref{fcpandeo}).

Our numerical experiments suggest that $s_{\bf 0}$ does not exist for all values of the index $\alpha \in (0,1]$. Specifically, we observe that for $\alpha \leq 0.526$, the function $y(x)$ lacks positive roots, whereas the opposite occurs for $\alpha \geq 0.527$. However, this is not a limitation for practical propoues, as for values of $\alpha$ relatively distant from $1$, that is, closer to $0$, the graph of $y(x)$ loses the symmetric appearance it exhibits in Figure \ref{FigB}. Furthermore, for values of $\alpha$ near $1$, it is advisable to consider function 
$$w(x)=\frac{y(x)+y(l-x)}{2}, \quad  x \in [0,l],$$
to achieve a symmetric figure.

The paper is structured as follows. Firstly, in Section \ref{FRC}, we derive the radius of curvature involving the fractional index $\alpha \in (0,1]$. In Section \ref{SecAEx}, Theorem \ref{Th1} presents the solution $y(x)$ of equation (\ref{Cde}). Assuming that $y(x)$ can be approximated by a polynomial, we establish in Theorem \ref{Th2} that $y(x)$ has a positive root. Additionally, Theorem \ref{Th3} provides a method for approximating such a root. In Section \ref{SecNE}, we demonstrate that for $\alpha = 0.526$, the function $y(x)$ does not have a positive root, whereas for $\alpha = 0.527$, it does. In Section \ref{SecCon} we conclude with some final comments.

\section{A fractional approximation of the radius of curvature} \label{FRC}

Let us start by remembering the definition of the Caputo fractional derivative, for $\beta\in (0,\infty) \backslash \mathbb{N}$. The left-side Caputo fractional derivative of a smooth function $y:[0,l]\rightarrow \mathbb{R}$ is defined as
\begin{equation*}
(D_{0+}^{\beta} y)(x) = \frac{1}{\Gamma (n-\beta)}\int_{0}^{x}\frac{y^{(n)}(t)}{(x-t)^{\beta-n+1}}dt,
\end{equation*}
where $n=[\beta]+1$ and $[\beta]$ denotes the integer part of $\beta$. It is known (see \cite{SKM} or \cite{KST}) that Caputo's fractional derivative is linear and satisfies
\begin{equation}
(D_{0+}^{\beta}t^{p})(x) = 
\begin{cases}
0, & p=0,...,n-1,\\
\frac{\Gamma(p+1)}{\Gamma (p+1-\beta)} x^{p-\beta}, & p>n-1.
\end{cases} \label{derpot}
\end{equation}

Let us assume the function $y:[0,l] \rightarrow \mathbb{R}$ has second derivative. Give a point $x\in (0,l)$, we know that the ratio of curvature, $\rho(x)$, of the osculating circle, $C(x)$, at $(x,y(x))$ is give by
\begin{equation}
\rho(x)=\frac{(1+(y^{\prime}(x))^{2})^{3/2}}{y^{\prime\prime}(x)}. \label{rccl}
\end{equation}
If the slope $y^{\prime}(x)$ is small enough that its square can be disregarded compared to $1$, then we have
\begin{equation}
\rho(x)\approx \frac{1}{y^{\prime\prime}(x)}.\label{aprccl}
\end{equation}

The traditional derivation of the radius of curvature relies on differentials, a concept not directly translatable to the fractional realm. To approximate (\ref{aprccl}) in this context, we adopt the following approach. Let $\theta(x)\in (-\pi,\pi)$ such that $y'(x)=\tan \theta(x)$. Then, the equation
$$L(z)=y(x)+y^{\prime}(x)(z-x), \ \ z\in (0,l),$$
describes the tangent line to $y$ at the point $P=(x, y(x))$. Now, consider a point $z$ close to $x$, and suppose, for example, $z < x$. Denote by $\theta(z)$ the angle between the segments associated with $z$ and $x$, and denote by $T(z)$ the tangent line at $(z, C(z))$ on the osculating circle (see Figure \ref{FigSe}).
\begin{figure}[ht]
\centering

\begin{tikzpicture}[scale=1]

\draw[->] (-2,1.5) -- (8.8,1.5);
\draw[->] (-1,3) -- (-1,-1.2);
\draw[thick,black, xshift=0cm] plot [smooth, tension=1] coordinates { (0.5,1.5) (4.7,0) (8,-0.2)};
\draw[red,-] (-1.5,1.5) -- (7.5,-0.65);
\draw[dashed,red] ([shift=(30:1cm)]1,1.2) arc (220:300:5cm);
\draw[teal] (8.7,0.9) node[anchor=north] {\footnotesize $C(x)$};
\draw (4.7,2) node[anchor=north] {\footnotesize $x$};
\draw[dashed] (4.7,1.5) -- (4.7,0.0);
\draw (4.7,0.2) node[anchor=north] {\tiny $\bullet$};
\draw[dashed,-] (4.7,0.0) -- (5.5,4);
\draw[teal] (5.9,3.5) node[anchor=north] {\scriptsize $\rho(x)$};
\draw[dashed,-] (5.5,4) -- (3.2,0.6);
\draw[cyan,-] (-0.5,2.65) -- (7,-1.56);
\draw (3.19,0.78) node[anchor=north] {\tiny $\bullet$};
\draw (3.19,2) node[anchor=north] {\footnotesize $z$};
\draw[dashed] (3.19,1.5) -- (3.19,0.6);
\draw[black] (8.4,-0.2) node {\scriptsize $y(x)$};
\draw[red] (8,-0.35) node[anchor=north] {\scriptsize $L(z)$};
\draw[thick,black] ([shift=(30:1cm)]5,-1.42) arc (-75:32:0.4cm);
\draw[black] (6.7,-1) node {\scriptsize $\theta(z)$};
\draw[thick,black] ([shift=(24:1cm)]3.85,2.5) arc (200:300:0.35cm);
\draw[black] (4.85,2.4) node {\scriptsize $\theta(z)$};
\draw (3.19,0.57) node[anchor=north] {\tiny $\bullet$};
\draw[black] (-1,-1.5) node {\scriptsize $y$};
\draw[black] (0.5,1.8) node {\scriptsize $0$};
\draw[cyan] (7.5,-1.5) node {\scriptsize $T(z)$};
\draw[black] (3,0.1) node {\scriptsize $P_{L}$};
\draw[black] (3.5,0.65) node {\scriptsize $P_{C}$};
\draw[black] (5,0.2) node {\scriptsize $P$};
\end{tikzpicture}
\caption{The osculating circle $C(x)$.} \label{FigSe}
\end{figure}
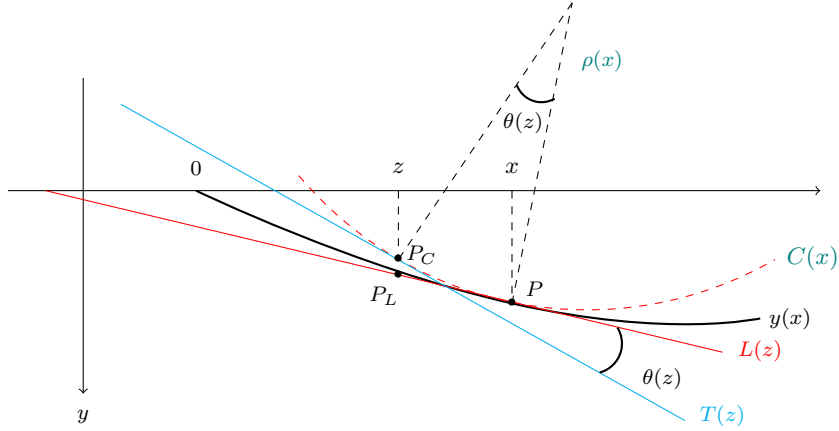

Let us denote by $S(z)$ the length of the arc corresponding to a circle with angle $\theta(z)$ and radius $\rho(x)$, then 
\begin{equation}
S(z)=\rho(x)\theta(z), \ \ z\in (0,l). \label{fpan}
\end{equation}
We adhere to the convention of considering the angle positive if it is measured in the clockwise direction. If $z$ approaches $x$ from the left, then $\theta(z)$ is negative. Given that $z$ is close to $x$, the points $P_{C}=(z, C(z))$ and $P_{L}=(z, L(z))$ are nearly coincident. Consequently, the magnitude of $S(z)$ is approximately the length of the segment from $P_{L}=(z, L(z))$ to $P=(x, y(x))$:
\begin{eqnarray}
S(z)&\approx&
\begin{cases}
-\sqrt{(z-x)^{2}+(L(z)-y(x))^{2}}, & 0<z<x,\\
\sqrt{(z-x)^{2}+(L(z)-y(x))^{2}}, & x<z<l.
\end{cases} \nonumber\\
&=&\sqrt{1+(y^{\prime}(x))^{2}}\cdot (z-x). \label{apxlarc}
\end{eqnarray}

Utilizing the linearity property of Caputo's fractional derivative and referencing equation (\ref{derpot}), with $0 < \alpha < 1$, $\beta = 1$, and $n = 1$, we derive
\begin{eqnarray}
D^{\alpha}_{0+}S(z)&\approx&\sqrt{1+(y^{\prime}(x))^{2}}\ D_{0+}^{\alpha} (z-x) \nonumber\\
&=&\frac{\sqrt{1+(y^{\prime}(x))^{2}}}{\Gamma(2-\alpha)} \cdot z^{1-\alpha}. \label{deriS}
\end{eqnarray}
Now, let us focus on the right-hand side of (\ref{fpan}), specifically addressing the function $\theta(z)$. Let $\theta(x)=\tan^{-1}(y^{\prime}(x))$ be the angle of the line $L$. Drawing a line parallel to the $x$-axis that passes through the point $I$, where the lines $L$ and $T$ intersect, we see that the angle of line $T$ is $\theta(x)+\theta(z)$. Let $R$ be the tangent line at the point $P_{z}=(z,y(z))$,
$$R(w)=y(z)+y^{\prime}(z)(w-z), \quad w\in (0,l).$$
Since $z\approx x$ and because the elastic curve $y(\cdot)$ is a smooth function, then $P_{z} \approx P_{C}$. Moreover, the angle of line $T$ is approximately equal to the angle of line $R$ (see Figure \ref{FigTL}):
\begin{equation}
\theta(x)+\theta(z) \approx \tan^{-1} (y^{\prime}(z)). \label{apxan}
\end{equation}
Since $z\approx x$, then $\theta(z)\approx 0$. We know that for small angles $x$, given in radians, $x \approx \tan x$. Using this and (\ref{apxan}) we deduce that
\begin{equation}
\theta(x)+\theta(z) \approx \tan (\theta(x)+\theta(z))\approx y^{\prime}(z). \label{apxptg}
\end{equation}
\begin{figure}[ht]
\centering

\begin{tikzpicture}[scale=1]

\draw[->] (-2,1.5) -- (8.8,1.5);
\draw[->] (-1,3) -- (-1,-1.2);
\draw[black, xshift=0cm, thick] plot [smooth, tension=1] coordinates { (0.5,1.5) (4.7,0) (8,-0.2)};
\draw[red,-] (-1.5,1.5) -- (7.5,-0.67);
\draw (4.7,2) node[anchor=north] {\footnotesize $x$};
\draw[dashed] (4.7,1.5) -- (4.7,0.0);
\draw (4.7,0.2) node[anchor=north] {\tiny $\bullet$};
\draw[cyan,-] (-0.5,2.65) -- (7,-1.56);
\draw (3.19,0.78) node[anchor=north] {\tiny $\bullet$};
\draw (3.19,2) node[anchor=north] {\footnotesize $z$};
\draw[dashed] (3.19,1.5) -- (3.19,0.6);
\draw[black] (8.4,-0.2) node {\scriptsize $y(x)$};
\draw[red] (8,-0.35) node[anchor=north] {\scriptsize $L(z)$};
\draw[thick,cyan] ([shift=(10:1cm)]1.78,0.635) arc (-55:35:0.5cm);
\draw[] (2.85,2.3) node {\scriptsize ${\color{red}\theta(x)}+{\color{cyan}\theta(z)}$};
\draw[thick,red] ([shift=(30:1cm)]-1.35,0.75) arc (-40:40:0.2cm);
\draw[red] (-0.3,0.7) node {\scriptsize $\theta(x)$};
\draw[thick,black] ([shift=(30:1cm)]5,-1.425) arc (-75:32:0.4cm);
\draw[black] (6.7,-1) node {\scriptsize $\theta(z)$};
\draw[blue,-] (-1.5,1.95) -- (7.5,-0.97);
\draw (3.19,0.62) node[anchor=north] {\tiny $\bullet$};
\draw[thick,blue] ([shift=(10:1cm)]0.2,0.91) arc (-40:29:0.37cm);
\draw[blue] (1.35,0.45) node {\scriptsize $\tan^{-1}(y^{\prime}(z))$};
\draw[black] (-1,-1.5) node {\scriptsize $y$};
\draw[cyan] (7.5,-1.5) node {\scriptsize $T(z)$};
\draw[black] (0.5,1.8) node {\scriptsize $0$};
\draw[blue] (8,-1.1) node {\scriptsize $R(w)$};
\draw[black] (3.2,0.0) node {\scriptsize $P_{z}$};
\draw[black] (3.5,0.65) node {\scriptsize $P_{C}$};
\draw[black] (3.8,-0.1) node {\scriptsize $I$};
\draw (3.85,0.4) node[anchor=north] {\tiny $\bullet$};
\draw[dashed] (-1.5,0.22) -- (8,0.22);
\end{tikzpicture}
\caption{The tangent lines. In blue is the tangent line at $(z,y(z))$.} \label{FigTL}
\end{figure}
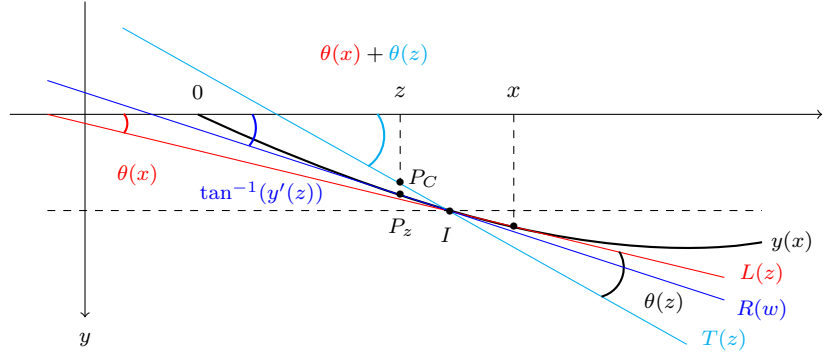
\noindent The observation that the fractional derivative of constant functions is zero, as indicated in equation (\ref{derpot}), results in the following equation
\begin{equation}
D^{\alpha}_{0+}\theta(z)=D^{\alpha}_{0+}(\theta(x)+\theta(z))\approx D^{\alpha}_{0+} y^{\prime}(z)=D^{\alpha+1}_{0+}y(z). \label{ddetheta}
\end{equation}

From (\ref{fpan}), we derive the following relationship
$$D^{\alpha}_{0+}S(z)|_{z=x} = \rho(x)D^{\alpha}_{0+}\theta(z)|_{z=x},$$
and by applying (\ref{deriS}) and (\ref{ddetheta}), we can deduce an analogue of (\ref{rccl})
$$\rho (x)= \frac{x^{1-\alpha}}{\Gamma(2-\alpha)} \cdot\frac{(1+(y^{\prime}(x))^{2})^{1/2}}{D^{\alpha+1}_{0+}y(x)}.$$
As in the previous case, if the square of $y^{\prime}(x)$ is significantly smaller than $1$, we arrive at the approximation
\begin{equation}
\rho(x) \approx \frac{x^{1-\alpha}}{\Gamma(2-\alpha)} \cdot \frac{1}{D^{\alpha+1}_{0+}y(x)}. \label{fraprc}
\end{equation}
This is similar to (\ref{aprccl}), but with the additional factor of $x^{1-\alpha}/\Gamma(2-\alpha)$.

\section{Solving the fractional differential equation}\label{SecAEx}

For $\alpha \in (0,1]$, we define the parameter $\lambda$ as follows
\begin{eqnarray*}
\lambda := \frac{P}{EI\,\Gamma(2-\alpha)},
\end{eqnarray*}
where $E$, $I$ and $P$ are positive constants. Our first aim is to solve the fractional differential equation
\begin{equation}
D_{0+}^{1+\alpha}y\left(x\right)=-\lambda x^{1-\alpha} \, y\left(x\right),\label{fDEq1}
\end{equation}
subject to the initial conditions
\begin{equation}
y(0)=b_{0}, \quad y^{\prime}(0)=b_{1}\label{fDEq2}.
\end{equation} 

Given the physical constraints of the problem, we have $b_{0}=0$. The initial condition for the derivative, denoted as $b_{1}$, remains undetermined. Identifying $b_{1}$ is crucial for calculating the critical load. We start by seeking the solution to the initial value problem (\ref{fDEq1})-(\ref{fDEq2}).

\begin{theorem} \label{Th1}
Let us set 
\begin{equation}
y\left(x\right) = b_{1} \left(a_{0} \, x+ \sum_{k=1}^{\infty}\left(-1\right)^{k}a_{k}\lambda^{k} x^{2k+1}\right) , \ \ x \in \mathbb{R}, \label{solFDE}
\end{equation}
where
$$a_{0}:=1, \quad a_{k}:=\prod_{i=1}^{k}\frac{\Gamma (2i+1-\alpha)}{\Gamma (2i+2)}, \ \ k\in \mathbb{N}.$$
The series converges absolutely and represents the solution to the initial value problem as stated in equations (\ref{fDEq1}) and (\ref{fDEq2}).
\end{theorem}

\begin{proof}
By applying the root criterion, we can ascertain that the power series in (\ref{solFDE}) converges absolutely over $\mathbb{R}$. This conclusion is drawn from the observation that
\begin{eqnarray*}
\limsup_{k\rightarrow \infty} \ (|(-1)^{k}a_{k}|)^{1/k}\leq \limsup_{k\rightarrow \infty} \ \frac{|(-1)^{k+1}a_{k+1}|}{|(-1)^{k}a_{k}|} = \limsup_{k\rightarrow \infty} \ \frac{\Gamma(2k+3-\alpha)}{\Gamma(2k+4)}=0.
\end{eqnarray*}
On the other hand, given that $y(0)=0$ and
\begin{equation*}
y^{\prime}\left(x\right)  =  b_{1}\left( 1+ \sum_{k=1}^{\infty} (-1)^{k} a_{k}\lambda^{k}\left(2k+1\right)x^{2k}\right),
\end{equation*}
yield $y^{\prime}(0)=b_{1}$. Now, let us calculate the fractional derivative of $y(x)$,
\begin{eqnarray*}
D_{0+}^{\alpha+1}y\left(x\right) = b_{1} \left( D_{0+}^{\alpha+1}x + \sum_{k=1}^{\infty}\left(-1\right)^{k}a_{k}\lambda^{k}D_{0+}^{\alpha+1}x^{2k+1}\right) .
\end{eqnarray*}
Since $\alpha+1\in\left(1,2\right]$, then $D_{0+}^{\alpha+1}x = 0$ (see (\ref{derpot})). Furthermore, for $k\ge 1$,
\begin{eqnarray*}
D_{0+}^{\alpha+1}x^{2k+1} = \frac{\Gamma(2k+2)}{\Gamma(2k+1-\alpha)}x^{2k-\alpha},
\end{eqnarray*}
from which we derive
\begin{eqnarray*}
D_{0+}^{\alpha+1}y\left(x\right)  & = & b_{1}\sum_{k=1}^{\infty}(-1)^{k}a_{k}\lambda^{k}\frac{\varGamma\left(2k+2\right)}{\varGamma\left(2k+1-\alpha\right)}x^{2k-\alpha}\\ 
 & = & b_{1}\left(-\lambda x^{2-\alpha}+\sum_{k=2}^{\infty}(-1)^{k}a_{k-1}\lambda^{k}x^{2k-\alpha}\right)\\
 & = & -\lambda x^{1-\alpha}b_{1}\left(x+\sum_{k=1}^{\infty}\left(-1\right)^{k}a_{k}\lambda^{k}x^{2k+1}\right)\\
 & = & -\lambda x^{1-\alpha}y(x).
\end{eqnarray*}
This result confirms the desired equality (\ref{fDEq1}).
\end{proof}

For $n \in \mathbb{N} \cup \{\infty\}$ we define
\begin{eqnarray*}
p_{n}(x) =\sum_{k=0}^{n}\left(-1\right)^{k}a_{k}x^{k}, \quad x\in \mathbb{R}.
\end{eqnarray*}
Then the solution $y(x)$ of (\ref{fDEq1}) can be expressed as
$$y(x)= b_{1} xp_{\infty}(\lambda x^{2}), \quad x\in [0,l].$$

We know that $y(l)=0$, then $b_{1}=0$ or $p_{\infty}(\lambda l^{2})=0$. If $b_{1}=0$, the column has not bent at all. On the other hand we have $p_{\infty}(\lambda l^{2})=0$. Thus we need to find the smallest positive root of the function $p_{\infty}$ in order to obtain the critical load (force).

We denote the smallest positive root of the polynomial $p_{n}$ as $r_{n}$. The existence of $r_{n}$ depends on whether $n$ is an odd or even integer. For odd integers $n$, a root $r_{n}$ always exists. This is due to the intermediate value theorem, as a polynomial of odd degree must cross the $x$-axis at least once, ensuring at least one real root. However, for even integers $n$, the existence of such a root is not guaranteed.

\begin{lemma} \label{ldescre}
Let $\alpha$ be in $(0,1]$. The following functions defined on $\mathbb{N}$
\begin{gather*}
f(x) = \frac{\Gamma(x+1)}{\Gamma(x-\alpha)}, \\
g(x) = \frac{x\Gamma(2x+4)}{(x+1)\Gamma(2x+3-\alpha)},
\end{gather*}
are both $f(x)$ and $g(x)$ strictly increasing functions.
\end{lemma}

\begin{proof}
To verify the monotonicity property, we will use the relation $\Gamma(x+1)=x\Gamma(x)$ for $x\in (0, \infty)$. Applying this, we have
 \begin{eqnarray*}
\frac{\Gamma(x+2)}{\Gamma(x+1-\alpha)}&=&\frac{(x+1)\Gamma(x+1)}{(x-\alpha)\Gamma(x-\alpha)}\\
 &\geq & \frac{x+1}{x} \cdot \frac{\Gamma(x+1)}{\Gamma(x-\alpha)}\\
 &= & \left(1+\frac{1}{x}\right) \cdot \frac{\Gamma(x+1)}{\Gamma(x-\alpha)}\\
 &>& \frac{\Gamma(x+1)}{\Gamma(x-\alpha)}
 \end{eqnarray*}
and
 \begin{eqnarray*}
 \frac{x+1}{x+2}\cdot \frac{\Gamma(2x+6)}{\Gamma(2x+5-\alpha)}&=&\frac{(x+1)^{2}(2x+5)(2x+4)}{x(x+2)(2x+4-\alpha)(2x+3-\alpha)}\cdot \frac{x\Gamma(2x+4)}{(x+1)\Gamma(2x+3-\alpha)}\\
 &\geq & \frac{(x+1)^{2}(2x+5)}{x(x+2)(2x+3)}\cdot \frac{x\Gamma(2x+4)}{(x+1)\Gamma(2x+3-\alpha)}\\ 
  &=&\left(1+\frac{2x+5}{2x^{3}+9x^{2}+10x}\right) \cdot \frac{x\Gamma(2x+4)}{(x+1)\Gamma(2x+3-\alpha)}\\
 &>& \frac{x\Gamma(2x+4)}{(x+1)\Gamma(2x+3-\alpha)}.
 \end{eqnarray*}
Thus, the desired property is clearly established.
\end{proof}

In what follows we require the following hypothesis.

\begin{description}
\item[Assumption A:] Let us suppose there is a polynomial $p_{n_{0}}$, with $n_{0}$ an even integer, such that it has a first positive root $r_{n_{0}}$ that satisfies the inequality
\begin{equation*}
\frac{\Gamma\left(2n_{0}+6\right)}{\Gamma\left(2n_{0}+5-\alpha\right)} \ge r_{n_{0}}.
\end{equation*}
\end{description}

\begin{lemma} \label{despp}
For each $n, m\in \mathbb{N}\cup\{0\}$ we have $p_{n_{0}+2n+1}<p_{\infty}<p_{n_{0}+2m}$ on $(0,r_{n_{0}}]$.
\end{lemma}

\begin{proof}
Given that the series below converges absolutely, we can rearrange the terms as follows
\begin{eqnarray*}
p_{\infty}(x)-p_{n_{0}+2n+1}(x)& = & \sum_{k=n_{0}+2n+2}^{\infty}(-1)^{k}a_{k}x^{k}\\
&=&\sum_{k=0}^{\infty}a_{n_{0}+2n+2k+2}\, x^{n_{0}+2n+2k+2}-\sum_{k=0}^{\infty}a_{n_{0}+2n+2k+3}\, x^{n_{0}+2n+2k+3}\\
&=& \sum_{k=0}^{\infty}a_{n_{0}+2n+2k+2}\, x^{n_{0}+2n+2k+2} \left(1-\frac{\Gamma(2n_{0}+4n+4k+7-\alpha)}{\Gamma(2n_{0}+4n+4k+8)}x \right).
\end{eqnarray*}
The monotonicity property of $f$, as stated in Lemma \ref{ldescre}, and the condition $x\in (0,r_{n_{0}}]$ imply
$$\frac{\Gamma(2n_{0}+4n+4k+8)}{\Gamma(2n_{0}+4n+4k+7-\alpha)}\geq \frac{\Gamma(2n_{0}+8)}{\Gamma(2n_{0}+7-\alpha)}>\frac{\Gamma(2n_{0}+6)}{\Gamma(2n_{0}+5-\alpha)}\geq r_{n_{0}}\geq x,$$
therefore, $p_{\infty}>p_{n_{0}+2n+1}$ on $(0,r_{n_{0}}]$. 

Now, let us consider
\begin{equation*}
p_{\infty}(x)-p_{n_{0}+2m}(x) = \sum_{k=0}^{\infty}a_{n_{0}+2m+2k+1}\, x^{n_{0}+2m+2k+1} \left(\frac{\Gamma(2n_{0}+4m+4k+5-\alpha)}{\Gamma(2n_{0}+4m+4k+6)}x-1 \right).
\end{equation*}
For $k\geq 1$ and $x\in (0,r_{n_{0}}]$, Lemma \ref{ldescre} implies
$$\frac{\Gamma(2n_{0}+4m+4k+6)}{\Gamma(2n_{0}+4m+4k+5-\alpha)} > \frac{\Gamma(2n_{0}+6)}{\Gamma(2n_{0}+5-\alpha)}\geq r_{n_{0}} \geq x,$$
thus $p_{n_{0}+2m}>p_{\infty}$ on $(0,r_{n_{0}}]$. This completes the proof of the desired inequalities.
\end{proof}

We now turn our attention to a result concerning the monotonicity of the sequence $(p_{n})_{n \in \mathbb{N}}$. This investigation will help us understand how the sequence behaves as $n$ varies, particularly in terms of its increasing or decreasing nature.

\begin{lemma} \label{lpoli}
Let $n$ and $m$ be two positive integers such that $n>m$. If $m$ and $n$ are  even integers, then $p_{n_{0}+n}<p_{n_{0}+m}$ on $(0,r_{n_{0}}]$. Additionally, if $m$ and $n$ are odd integers, then $p_{n_{0}+n}>p_{n_{0}+m}$ on $(0,r_{n_{0}}]$.
\end{lemma}

\begin{proof}
If $n$ and $m$ are even integers, with $n>m$, then 
\begin{eqnarray*}
p_{n_{0}+n}(x)-p_{n_{0}+m}(x) & = &\sum_{k=n_{0}+m+1}^{n_{0}+n}(-1)^{k}a_{k}x^{k}\\
&=& \sum_{k=0}^{(n-m-2)/2}a_{2k+n_{0}+m+1} \,x^{2k+n_{0}+m+1} \left(\frac{\Gamma(2n_{0}+2m+4k+5-\alpha)}{\Gamma(2n_{0}+2m+4k+6)}x-1\right).
\end{eqnarray*}
For $x\in (0,r_{n_{0}}]$ we have, from Lemma \ref{ldescre},
$$\frac{\Gamma(2n_{0}+2m+4k+6)}{\Gamma(2n_{0}+2m+4k+5-\alpha)}\geq \frac{\Gamma(2n_{0}+10)}{\Gamma(2n_{0}+9-\alpha)} > \frac{\Gamma(2n_{0}+6)}{\Gamma(2n_{0}+5-\alpha)}\geq r_{n_{0}} \geq x.$$
Therefore $p_{n_{0}+m}>p_{n_{0}+n}$ on $(0,r_{n_{0}}]$. 

Now, let $m$ and $n$ be two odd integers such that $n>m$, then 
\begin{equation*}
p_{n_{0}+n}(x)-p_{n_{0}+m}(x)  =  \sum_{k=0}^{(n-m-2)/2}a_{n_{0}+2k+m+1}\,x^{n_{0}+2k+m+1} \left(1-\frac{\Gamma(2n_{0}+4k+2m+5-\alpha)}{\Gamma(2n_{0}+4k+2m+6)}x \right).
\end{equation*}
Again, by Lemma \ref{ldescre} we have, for $x\in (0,r_{n_{0}}]$,
$$\frac{\Gamma(2n_{0}+4k+2m+6)}{\Gamma(2n_{0}+4k+2m+5-\alpha)}\geq \frac{\Gamma(2n_{0}+10)}{\Gamma(2n_{0}+9-\alpha)} > \frac{\Gamma(2n_{0}+6)}{\Gamma(2n_{0}+5-\alpha)}\geq r_{n_{0}} \geq x,$$
therefore $p_{n_{0}+m}<p_{n_{0}+n}$ on $(0,r_{n_{0}}]$. This proves the result.
\end{proof}

Note that $p_{n_{0}+1}(0)=1$ and
$$p_{n_{0}+1}(r_{n_{0}})=p_{n_{0}}(r_{n_{0}})-a_{n_{0}+1}\, x^{n_{0}+1}<0,$$
since $p_{n_{0}}(r_{n_{0}})=0$. Therefore, by the intermediate value theorem, there exists a root $r_{n_{0}+1}$ such that $0<r_{n_{0}+1}<r_{n_{0}}$. We are now equipped to demonstrate the existence of a positive root for $p_{\infty}$. 

\begin{theorem} \label{Th2}
The function $p_{\infty}$ has a first positive root, and this root is located in $(r_{n_{0}+1},r_{n_{0}})$.
\end{theorem}

\begin{proof}
By taking $m=0$ in Lemma \ref{despp}, we deduce that $p_{\infty}(r_{n_{0}})<p_{n_{0}}(r_{n_{0}})=0$. Given that $p_{\infty}(0)=1$,  the intermediate value theorem guarantees that $p_{\infty}$ must have a root in the interval $(0,r_{n_{0}})$. Let us denote by 
$s_{\bf 0}$ the first positive root of $p_{\infty}$. 

For each $n\in \mathbb{N} \cup \{0\}$, Lemma \ref{despp} implies that $p_{n_{0}+2n+1}(s_{\bf 0})<p_{\infty}(s_{\bf 0})=0$. Therefore, $p_{n_{0}+2n+1}$ has a root $r_{n_{0}+2n+1}$, and it follows that $r_{n_{0}+2n+1}<s_{\bf 0}$. Consequently, we obtain a sequence $(r_{n_{0}+2n+1})_{n\in \mathbb{N} \cup \{0\}}$. This sequence is increasing; indeed, Lemma \ref{lpoli} indicates that  $p_{n_{0}+2n+1}(r_{n_{0}+2n+3})<p_{n_{0}+2n+3}(r_{n_{0}+2n+3})=0$, implying that $r_{n_{0}+2n+1}<r_{n_{0}+2n+3}$. Let us define
$$r:=\lim_{n\rightarrow \infty} r_{n_{0}+2n+1}=\sup_{n\in \mathbb{N} \cup \{0\}}\{r_{n_{0}+2n+1}\}\leq s_{\bf 0}.$$
We aim to show that $r=s_{\bf 0}$. To do this, suppose that $r<s_{\bf 0}$. Take an $\tilde{r}\in (r,s_{\bf 0})$. By the definition of $s_{\bf 0}$ we have $p_{\infty}(x)>0$ for $x\in [r_{n_{0}+1}, \tilde{r}]$. The continuity of $p_{\infty}$ implies 
$$\min_{z\in [r_{n_{0}+1},\tilde{r}]}\{p_{\infty}(z)\}>0.$$
By Weierstrass's M-test, the sequence $(p_{n})_{n\in \mathbb{N}}$ converges uniformly over compact intervals. Hence, there exists $n_{1}\in \mathbb{N}$ such that for all $n\geq n_{1}$,
$$|p_{\infty}(x)-p_{n}(x)|<\frac{1}{2} \min_{z\in [r_{n_{0}+1},\tilde{r}]}\{p_{\infty}(z)\}, \quad x\in  [r_{n_{0}+1},\tilde{r}].$$
In particular, we can choose a sufficiently large odd integer $n$ such that $r_{n}\in (r_{n_{0}+1},r)$. The preceding inequality then leads to
$$0=p_{n}(r_{n})>\frac{1}{2} \, p_{\infty}(r_{n})>0.$$
This contradiction implies the desired equality, $r=s_{\bf 0}$.
\end{proof}

We now propose a method to approximate the first positive root of $p_{\infty}$. To this end, we introduce the following hypothesis.

\begin{description}
\item[Assumption B:] 
We assume that $p^{\prime}_{n_{0}}(x)\leq 0$ for $x$ on $[r_{n_{0}+1},r_{n_{0}}]$ and that
$$\frac{n_{0}+1}{n_{0}+2}\cdot \frac{\Gamma(2n_{0}+6)}{\Gamma(2n_{0}+5-\alpha)}\geq r_{n_{0}}.$$
\end{description}

\begin{theorem}  \label{Th3}
If $n\geq n_{0}$, regardless of whether $n$ is even or odd, the existence of the first root, $r_{n}$, of $p_{n}$ is guaranteed. Furthermore, the sequence $(r_{n})_{n\in \mathbb{N}}$ satisfies the condition
\begin{equation}
\lim_{n \rightarrow \infty}(r_{n+1}-r_{n})=0. \label{elim}
\end{equation}
Given an $\varepsilon>0$ we can find an integer $\tilde{n}$ (existence ensured by the first part of this theorem) such that
\begin{equation}
\left|r_{\tilde{n}}-r_{\tilde{n}+1}\right|<\varepsilon \label{apxdif}
\end{equation}
then 
\begin{equation}
\left|s_{\bf 0}-r_{\tilde{n}}\right|<\varepsilon. \label{aproot}
\end{equation}
Consequently, this provides an approximation to the first positive root, $s_{\bf 0}$, of $p_{\infty}$ with an error margin of at most $\varepsilon$ units.
\end{theorem}

\begin{proof}
For each $m\in \mathbb{N} \cup \{0\}$, we have $p_{n_{0}+2m}(0)=1$. Note that Lemma \ref{lpoli} implies $p_{n_{0}+2m}(r_{n_{0}})<p_{n_{0}}(r_{n_{0}})=0$. Therefore, $p_{n_{0}+2m}$ has a minimal positive root, $r_{n_{0}+2m}$, for all $m\in \mathbb{N} \cup \{0\}$. Additionally, Lemma \ref{despp} indicates that $p_{\infty}(x)<p_{n_{0}+2m}(x)$ in $(0,r_{n_{0}}]$ and since $p_{\infty}(s_{\bf 0})=0<p_{n_{0}+2m}(s_{\bf 0})$, it follows that $s_{\bf 0}<r_{n_{0}+2m}$. Furthermore, according to Lemma \ref{lpoli} we have $p_{n_{0}+2m+2}(r_{n_{0}+2m})<p_{n_{0}+2m}(r_{n_{0}+2m})=0$, implying $r_{n_{0}+2m+2}<r_{n_{0}+2m}$. Consequently, the sequence $(r_{n_{0}+2m})_{m}$ is monotonically decreasing and bounded below by  $s_{\bf 0}$. This leads to the conclusion that
$$ r:=\lim_{m\rightarrow \infty} r_{n_{0}+2m}=\inf_{m\in \mathbb{N}\cup \{0\}} \{r_{n_{0}+2m}\}\geq s_{\bf 0}.$$
Now, let us demonstrate that $s_{\bf 0}=r$. Suppose, for contradiction, that $s_{\bf 0}<r$. First, let us consider the expression
\begin{equation*}
p^{\prime}_{\infty}(x)-p^{\prime}_{n_{0}}(x) 
= \sum_{k=0}^{\infty}a_{n_{0}+2k}\, x^{n_{0}+2k+1} \left((n_{0}+2k+2)\frac{\Gamma(2n_{0}+4k+5-\alpha)}{\Gamma(2n_{0}+4k+6)}x-(n_{0}+2k+1)\right).
\end{equation*}
Using the monotonicity property of $g$ as established in Lemma \ref{ldescre}, for $x\in (0,r_{n_{0}}]$ and $k\geq 1$, we have
\begin{equation*}
\frac{n_{0}+2k+1}{n_{0}+2k+2}\cdot \frac{\Gamma(2n_{0}+4k+6)}{\Gamma(2n_{0}+4k+5-\alpha)} > 
\frac{n_{0}+1}{n_{0}+2}\cdot \frac{\Gamma(2n_{0}+6)}{\Gamma(2n_{0}+5-\alpha)}\geq r_{n_{0}}\geq x,
\end{equation*}
which implies that $p^{\prime}_{\infty}(s_{\bf 0})<p^{\prime}_{n_{0}}(s_{\bf 0})\leq 0$. The continuity of $p^{\prime}_{\infty}$ implies the existence of a $\delta > 0$ such that $p_{\infty}(x)<p_{\infty}(s_{\bf 0})=0$, for all $s_{\bf 0}\leq x < s_{\bf 0}+ \delta < r$. Let us take an $\tilde{x} \in (s_{\bf 0},s_{\bf 0}+\delta)$. Since $\lim_{n\rightarrow \infty}p_{n}(\tilde{x})=p_{\infty}(\tilde{x})$ we can find a large enough even integer $n$ such that
$$|p_{\infty}(\tilde{x})-p_{n}(\tilde{x})|<-p_{\infty}(\tilde{x}),$$
which implies that $p_{n}(\tilde{x})<0$. Given that $p_{n}(0)=1$  and by applying the intermediate value theorem, it follows that  $r_{n}<s_{\bf 0} + \delta < r$. However, this conclusion is contradictory, as we have established that $r\leq r_{n}$ for all even integer $n$. Consequently, it must be that $s_{\bf 0}=r$.

Given $\varepsilon>0$, there exists an integer $\tilde{n}$, such that if $i\geq \tilde{n}$ and  $j\geq \tilde{n}$, with $i$ an odd integer and $j$ an even integer, then 
$$|r_{i}-s_{\bf 0}|<\frac{\varepsilon}{2}, \quad |r_{j}-s_{\bf 0}|<\frac{\varepsilon}{2}.$$
Consequently, if $n\geq \tilde{n}$, we have
$$|r_{n+1}-r_{n}|<|r_{n+1}-s_{\bf 0}|+|r_{n}-s_{\bf 0}|<\varepsilon.$$
Thus, the limit as expressed in (\ref{elim}) is effectively established.
\end{proof}

In this context, we have find an approximation, $r_{\tilde{n}}$, of $s_{\bf 0}$. Therefore, if  $p_{\infty}(\lambda l^{2})=0$, it follows that
$$\frac{P}{EI\,\Gamma(2-\alpha)} \ l^{2} = s_{\bf 0}.$$
From this, we obtain an approximation critical load, see the formula (\ref{fcpandeo}).

\section{Some illustrative examples} \label{SecNE}

In this section, we will explore two illustrative examples to further elucidate the concepts.

\subsection{Non-existence of an even degree polynomial $p_{n_{0}}$ with a positive root}

In this subsection, we will take $\alpha=0.526$ and $n\in \mathbb{N}\cup \{0\}$. Observe that
\begin{equation*}
p_{\infty}(x)-p_{2n+1}(x)
= \sum_{k=0}^{\infty}a_{2n+2k+2}\, x^{2n+2k+2} \left(1-\frac{\Gamma(4n+4k+6.474)}{\Gamma(4n+4k+8)}x \right).
\end{equation*}
According to Lemma \ref{ldescre}, for $k\geq 1$,
$$\frac{\Gamma(4n+4k+8)}{\Gamma(4n+4k+6.474)} > \frac{\Gamma(4n+8)}{\Gamma(4n+6.474)},$$
therefore, $p_{\infty}>p_{2n+1}$ on $[0,\Gamma(4n+8)/\Gamma(4n+6.474)]$. This particularly implies that if
\begin{equation}
r_{2n+1}\in [0,\Gamma(4n+8)/\Gamma(4n+6.474)], \label{cplab}
\end{equation}
then $p_{\infty}$ does not have a root in $[0,r_{2n+1}]$. 

The Table \ref{T1} displays the first positive root $r_{m}$ of the polynomial $p_{m}$ along with the corresponding evaluation of the term  $\Gamma (2m+6)/\Gamma (2m+4.474)$. 
\begin{table}[h!]
\centering
\begin{tabular}{|c|c|c|c|c|c|c|c|c|c|c|c|}
\hline 
$p_{m}$ & $p_{1}$ & $p_{3}$ & $p_{5}$ & $p_{7}$ & $p_{9}$ & $p_{11}$ & $p_{13}$ & $p_{15}$& $p_{17}$ & $p_{19}$ & $p_{21}$\tabularnewline
\hline 
$r_{m}$ & 4.56 & 7.93 & 11.31 & 14.37 & 16.68 & 29.93 & 38.25 & 46.25 & 54.5 & 63.31 & 72.87 \tabularnewline
\hline
$\frac{\Gamma(2m+6)}{\Gamma\left( 2m+4.474\right) }$ & 18.3 & 37.3 & 60.6 & 87.5 & 117.5 & 150.5 & 186.2 & 224.5 & 265.1 & 308.0 & 353.1  \tabularnewline
\hline 
\end{tabular}
\caption{First positive root of the polynomial $p_{m}$.} \label{T1}
\end{table}

From Table \ref{T1}, we observe that the condition (\ref{cplab}) is fulfilled in each instance. Consequently, it can be inferred that $p_{\infty}$ has no roots within the interval $[0, r_{m}]$. Moreover, we notice that the root $r_{m}$ increases with the degree $m$ of the odd polynomial $p_{m}$. This observation, in conjunction with the graph of $p_{\infty}$ presented in Figure \ref{Gra526}, provides numerical evidence supporting the conclusion that $p_{\infty}$ lacks a positive root when $\alpha = 0.526$. In fact, for values of $\alpha$ less than $0.526$, the function $p_{\infty}$ also appears to be devoid of positive roots, as exemplified by the graphs of $p_{\infty}$ in Figure \ref{Gra14}, corresponding to $\alpha = 0.1$ and $\alpha = 0.4$.
\begin{figure}[!ht]   
    \begin{subfigure}[t]{.47\textwidth}
    \includegraphics[width=3.3in,height=2.5in]{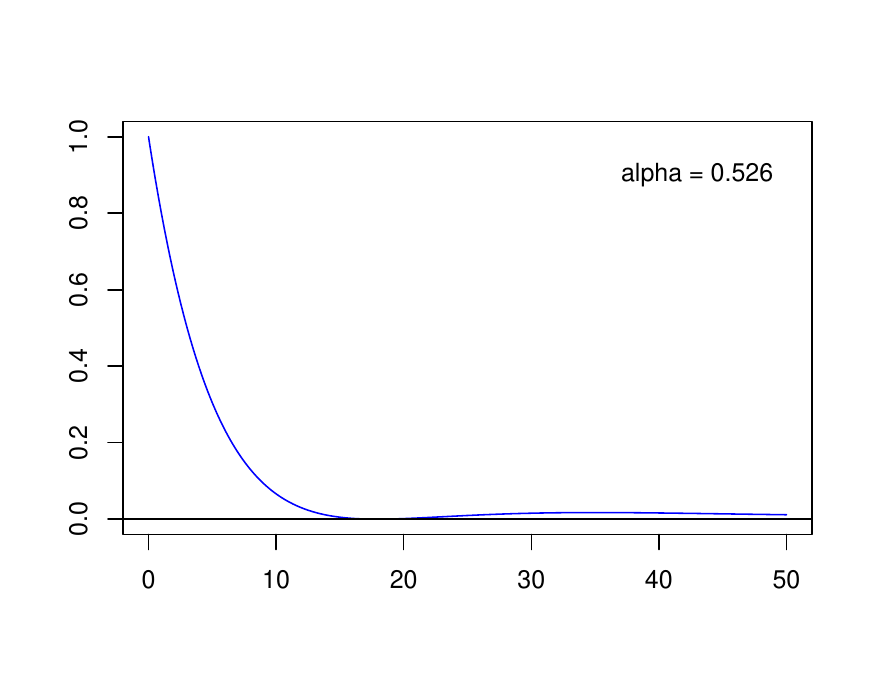}
    \end{subfigure}
    \hspace{0.0cm}
     \begin{subfigure}[t]{.47\textwidth}
    \includegraphics[width=3.3in,height=2.5in]{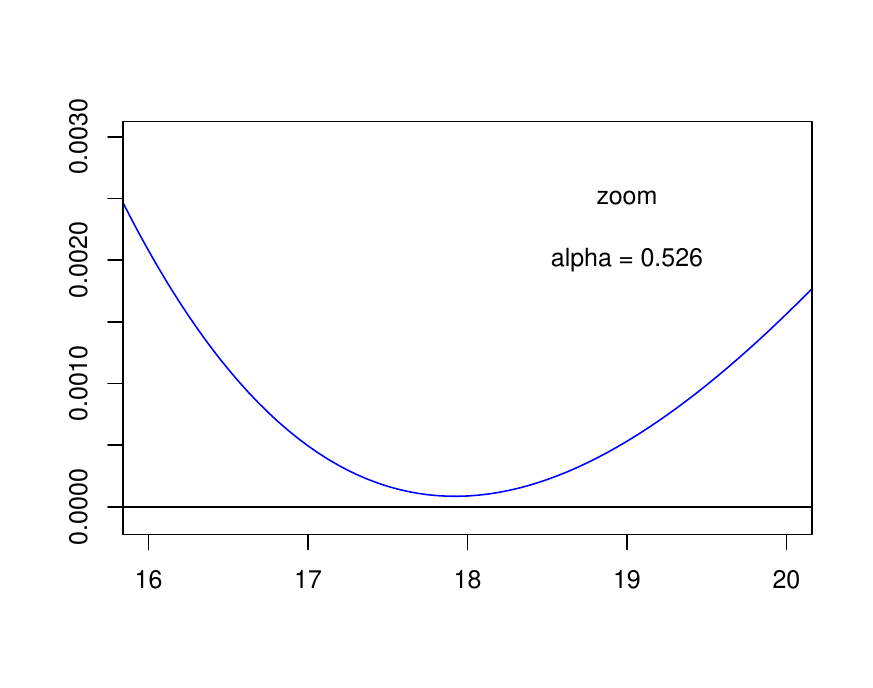}
    \end{subfigure}
 \vspace{-0.5cm}
   \caption{Graph of $p_{\infty}(x)$ for $\alpha= 0.526$.}   \label{Gra526} 
\end{figure}

\begin{figure}[!ht]   
 \begin{subfigure}[t]{.47\textwidth}
    \includegraphics[width=3.3in,height=2.5in]{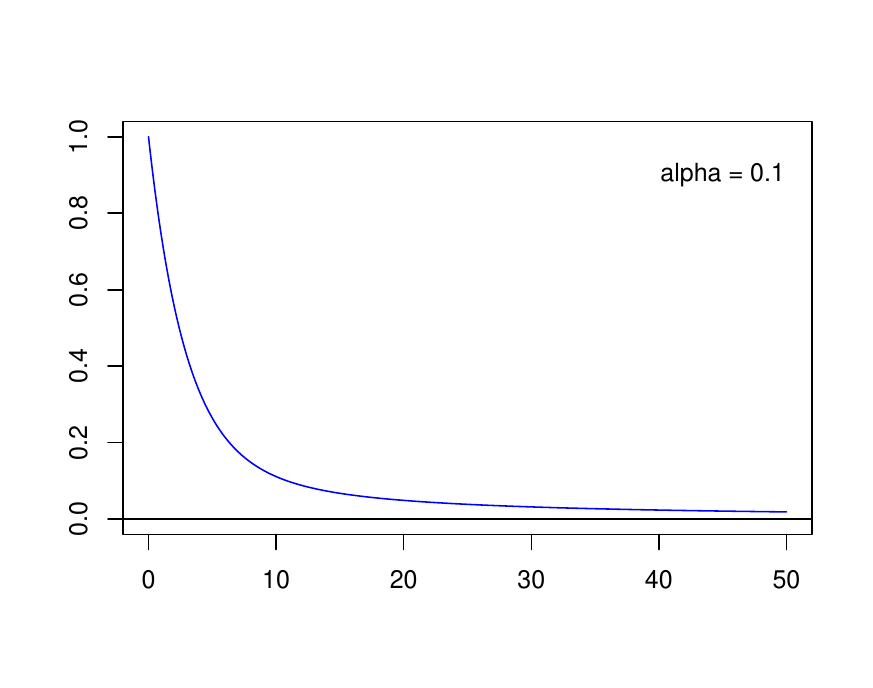}
    \end{subfigure}
     \begin{subfigure}[t]{.47\textwidth}
    \includegraphics[width=3.3in,height=2.5in]{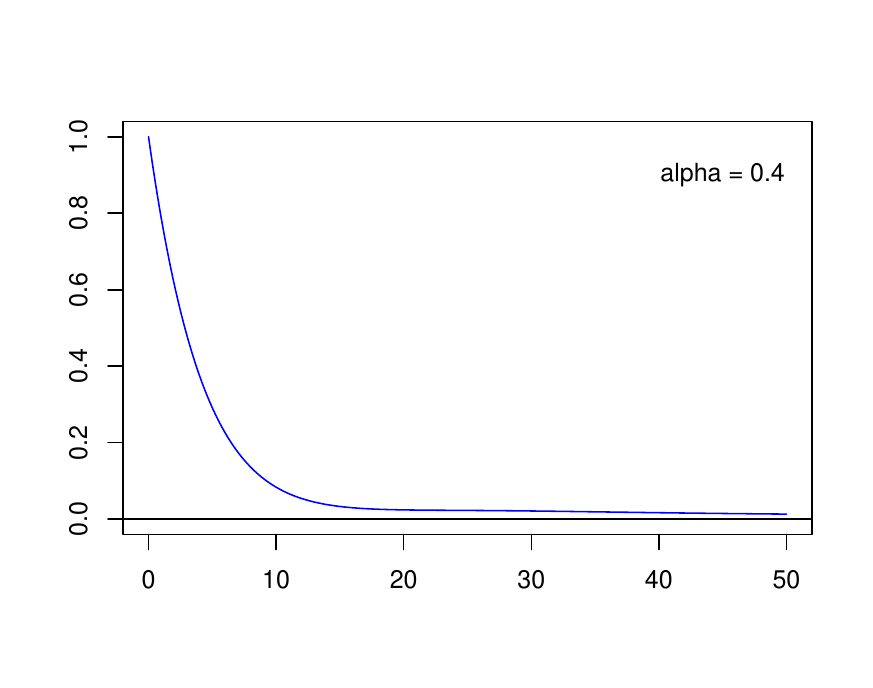}
    \end{subfigure}
 \vspace{-0.5cm}
   \caption{Graphs of $p_{\infty}(x)$ for $\alpha = 0.1$ and $\alpha = 0.4$.}   \label{Gra14} 
\end{figure}

\subsection{Existence of an even degree polynomial $p_{n_{0}}$ with a positive root}

\begin{figure}[!ht]   
    \begin{subfigure}[t]{.47\textwidth}
    \includegraphics[width=3.3in,height=2.5in]{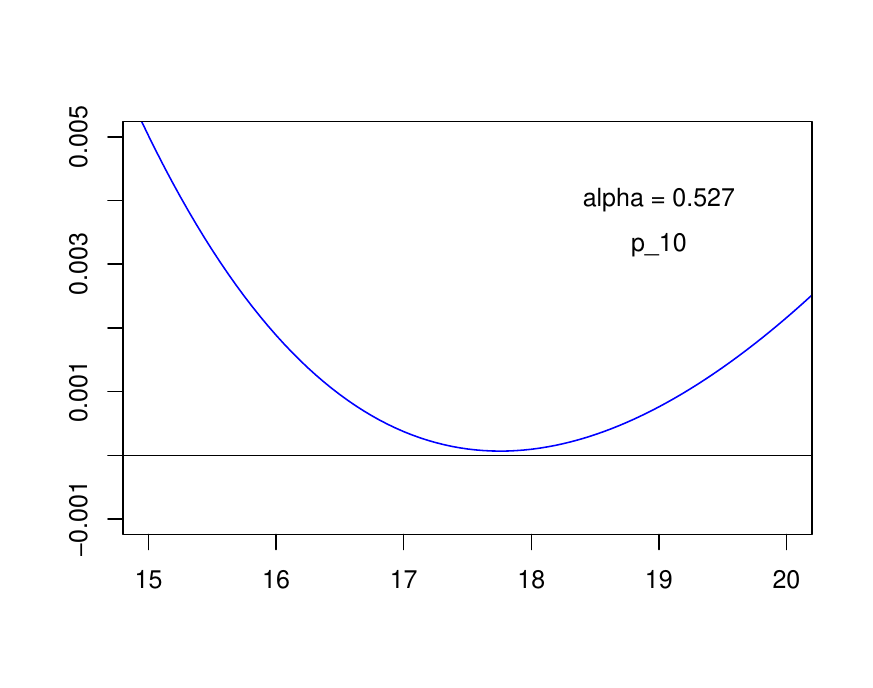}
    \end{subfigure}
    \hspace{.5cm}
     \begin{subfigure}[t]{.47\textwidth}
    \includegraphics[width=3.3in,height=2.5in]{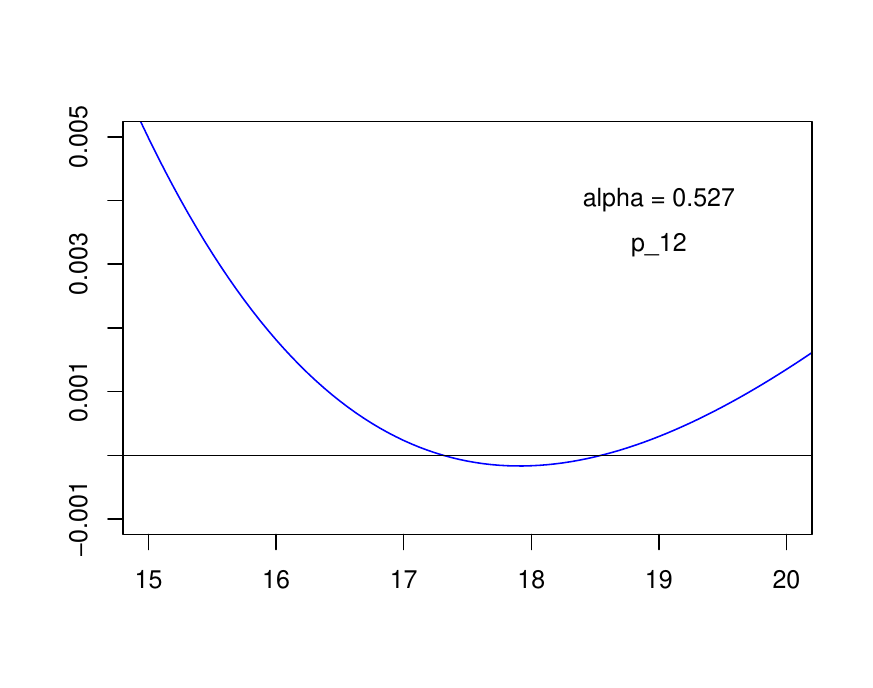}
    \end{subfigure}
 \vspace{-0.5cm}
   \caption{Graphs of polynomials $p_{10}(x)$ and $p_{12}(x)$.}   \label{Gp1012} 
\end{figure}

Now, we will consider $\alpha = 0.527$. Polynomials $p_{2}$, $p_{4}$, $p_{6}$, $p_{8}$ and $p_{10}$ do not have positive roots. However, polynomial $p_{12}$ does have a positive root, $r_{12}\approx 17.31$, as shown in Figure \ref{Gp1012}. Let us set $n_{0}=12$. Then 
\begin{equation}
168.65 \approx \frac{\Gamma(2n_{0}+5)}{\Gamma(2n_{0}+5-\alpha)}\geq r_{n_{0}} \approx 17.31. \label{dp}
\end{equation}

The Table \ref{T2} presents the roots of the polynomials $p_{12}$, $p_{13}$, $p_{14}$ and $p_{15}$ for various values of $\alpha$.
\begin{table}[h!]
\centering
\begin{tabular}{|c|c|c|c|}
\hline 
$r_{m} \diagdown \alpha $ & $0.527$ & $0.6$ & $0.7$ \tabularnewline
\hline 
$r_{12}$ & 17.31407406 & 12.97715575 & 11.45145303 \tabularnewline
\hline 
$r_{13}$ & 17.30910486 & 12.97715515 & 11.45145302  \tabularnewline
\hline 
$r_{14}$ & 17.30961111 & 12.97715518 & 11.45145302  \tabularnewline
\hline
$r_{15}$ & 17.30956418 & 12.97715518 & 11.45145302  \tabularnewline
\hline
\end{tabular}
\caption{First positive root of $p_{12}$, $p_{13}$, $p_{14}$ and $p_{15}$.} \label{T2}
\end{table}

From this, we observe that $r_{n_{0}+1} \approx 17.3$. The inequality (\ref{dp}) confirms that  Assumption A is satisfied. Consequently, based on Theorem \ref{Th1}, it follows that the first positive root of $p_{\infty}$ lies within the interval $(17.3,17.31)$.\\

Now, let's assume we aim to find the root of $p_{\infty}$ with an error margin less than $\varepsilon=1\times 10^{-3}$. To achieve this, it is necessary to ensure that Assumption B is fulfilled. Observing the graph of  $p_{12}$ (see Figure \ref{Gp1012}) we can infer that  $p_{12}$ is decreasing in the interval $(16,17.5)$. Consequently, $p_{12}^{\prime}<0$ on $[r_{13},r_{12}]\approx [17.3,17.31]$, and 
$$156.6 \approx \frac{n_{0}+1}{n_{0}+2} \cdot \frac{\Gamma(2n_{0}+6)}{\Gamma(2n_{0}+5-\alpha)} \geq r_{n_{0}} \approx 17.31,$$
thus Assumption B is satisfied. Recall that $\alpha = 0.527$ and 
$$|r_{14}-r_{15}|=4.693\times 10^{-5}< 1 \times 10^{-3},$$
which, according to Theorem \ref{Th2}, implies that
\[
|s_{\bf 0}-r_{14}|< 1 \times 10^{-3}.
\]
Therefore, $r_{14} = 17.30961111$ serves as the desired approximation.\\

The function $p_{\infty}$ possesses a positive root when the value of $\alpha$ exceeds $0.527$. This assertion is supported by the graphical representations of $p_{\infty}$ for $\alpha = 0.6$ and $\alpha = 0.8$, as depicted in Figure \ref{G68}. Additionally, Table \ref{T2} illustrates that the root $r_{m}$ tends to stabilize rapidly with increasing values of $\alpha$. Consequently, the root of $p_{12}$  serves as an excellent approximation for the root of $p_{\infty}$ when $\alpha$ is greater than $0.6$ or $0.7$.
\begin{figure}[!ht]   
    \begin{subfigure}[t]{.47\textwidth}
    \includegraphics[width=3.3in,height=2.5in]{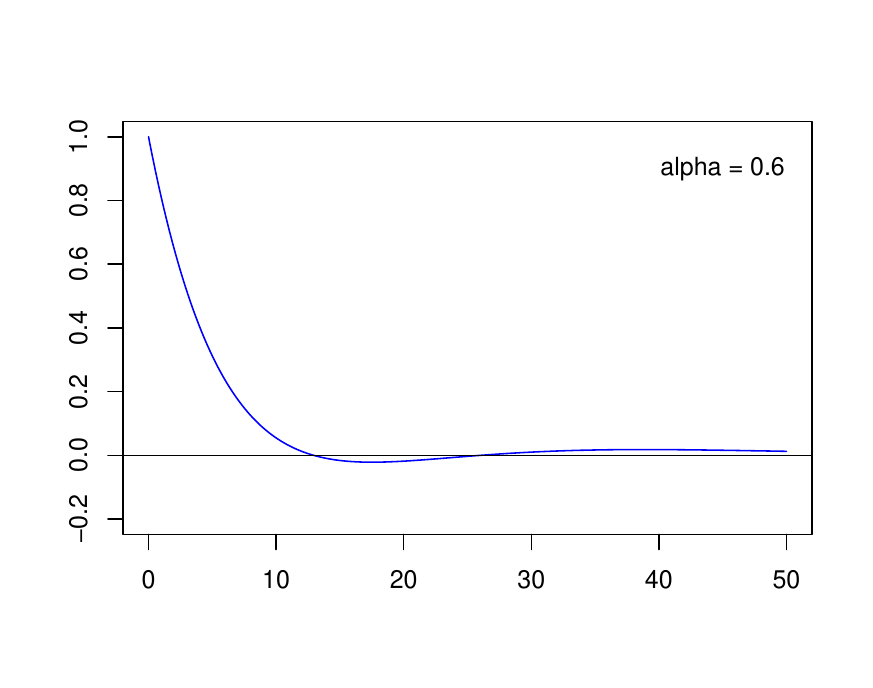}
    \end{subfigure}
    \hspace{.5cm}
     \begin{subfigure}[t]{.47\textwidth}
    \includegraphics[width=3.3in,height=2.5in]{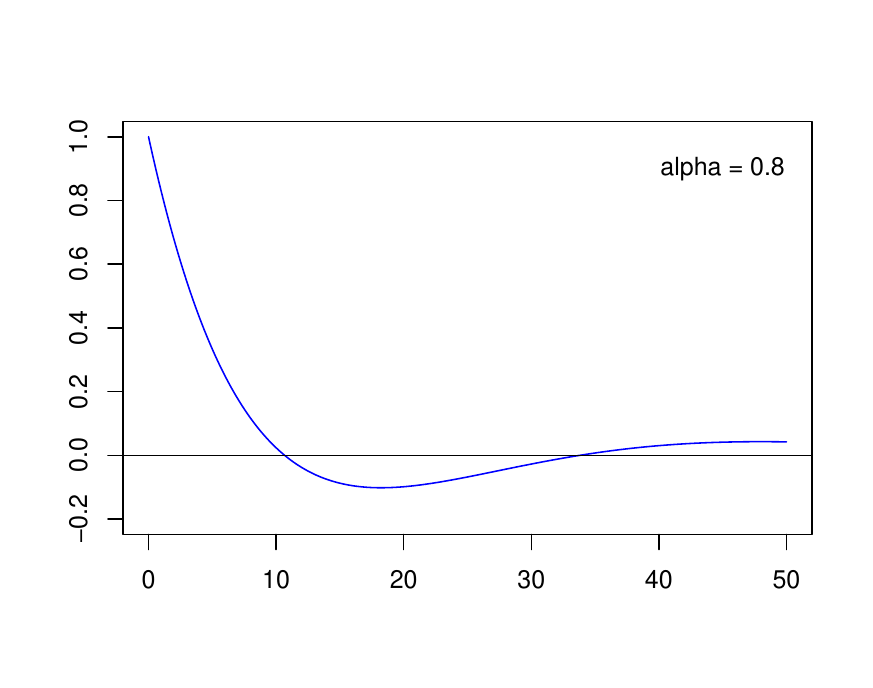}
    \end{subfigure}
 \vspace{-0.5cm}
   \caption{Graphs of $p_{\infty}(x)$ for $\alpha=0.6$ and $\alpha=0.8$.}   \label{G68} 
\end{figure}

\section{Conclusions} \label{SecCon}

In this section, we will address the results obtained. Initially, we have deduced a formula for the radius of curvature that incorporates the fractional index $\alpha \in (0,1]$. This formula has been pivotal in formulating our fractional model for column buckling. Subsequently, based on the hypothesis that a polynomial can accurately approximate the elastic curve, we have established the existence of a critical buckling force (Theorem \ref{Th2}). Furthermore, we have managed to calculate an approximation of this critical buckling force that, within a previously defined margin of error (Theorem \ref{Th3}), closely approaches the real force. We observe that for values of $\alpha \leq 0.526$, the critical force does not exist, while for $\alpha \geq 0.527$, it exists.

Expression (\ref{EEq}) reveals substantial potential for diversification. By merely altering the curvature moment of the structure of interest, we can derive various fractional differential equations. The real challenge lies in proving that the solution to the fractional differential equation reaches a positive minimum and in developing an effective numerical method for its calculation, as we have demonstrated in this work.

Another important area of opportunity is the exploration of alternative models for beam buckling. This is a vast field, as there are numerous ways in which a beam can buckle (see \cite{TG}). Moreover, for each model, different types of fractional derivatives (beyond the Caputo derivative) can be considered, allowing for physical verification of which derivative best captures the buckling behavior.

It is important to acknowledge several limitations of this study. The first is that the equipment required for the physical validation of the fractional model is very expensive, at least for us. However, other researchers may have access to such equipment or the necessary resources to continue developing our work. The second limitation is theoretical: the initial approximations to ${\bf s_{0}}$ are determined by graphically analyzing the roots of the polynomials $p_{even}$ (see Assumption A).

\subsection*{Acknowledgment}
The authors JVM and MRA received partial support from grants PIM22-1 and PIM23-6, respectively, from the Universidad Aut\'onoma de Aguascalientes.

\bibliographystyle{elsarticle-harv}
\bibliography{Vigas2_v2}

\end{document}